\documentclass[a4paper,10pt]{amsart}
\usepackage{amsmath,amsthm,latexsym,amscd,amssymb}
\usepackage[all]{xy}
\makeindex
\numberwithin{equation}{section}

\newtheorem{prop}{Proposition}[section]

\newtheorem{ddd}[prop]{Definition}

\newcommand{\End}{\mathop{\mbox{\rm End}}}
\newcommand{\Ind}{{\rm ind}}
\newcommand{\dom}{\mathop{\rm dom}}

\newcommand{\M}{{\mathcal M}}
\newcommand{\Fi}{{\mathfrak F}{\mathfrak i}{\mathfrak n}}

\newcommand{\st}{\star_{\alpha}}
\newcommand{\Ki}{{\rm K}}
\newcommand{\KKi}{{\rm KK}}
\newcommand{\Ft}{{\mathcal F}}
\newcommand{\dira}{\partial\!\!\!/}
\newcommand{\gl}{{\mathfrak g}}

\newcommand{\incl}{\hookrightarrow}
\newcommand{\Tr}{{\rm Tr}}
\newcommand{\ev}{{\rm ev}}

\newcommand{\Alpha}{\mathfrak A}

\newcommand{\N}{{\mathcal N}}

\newcommand{\D}{{\mathcal D}}
\newcommand{\U}{{\mathcal U}}
\newcommand{\R}{{\mathcal R}}

\newcommand{\A}{{\mathcal A}}
\newcommand{\B}{{\mathcal B}}
\newcommand{\Cg}{{\mathcal C}}
\newcommand{\Ai}{{\mathcal A}_{\infty}}

\newcommand{\C}{C^{\infty}}

\newcommand{\K}{{\mathcal K}}

\newcommand{\lr}{\longrightarrow}

\newcommand{\ten}{\otimes}

\newcommand{\ve}{\varepsilon}
\newcommand{\ov}{\overline}

\newcommand{\Dir}{D \!\!\!\!/\,}

\DeclareMathOperator{\spfl}{sf}
\DeclareMathOperator{\Ran}{Ran}
\DeclareMathOperator{\Ker}{Ker}

\def\bbbr{{\rm I\!R}} 
\def\bbbn{{\rm I\!N}} 

\def\bbbc{{\rm I\!C}}

\def\bbbq{{\mathchoice {\setbox0=\hbox{$\displaystyle\rm Q$}\hbox{\raise
0.15\ht0\hbox to0pt{\kern0.4\wd0\vrule height0.8\ht0\hss}\box0}}
{\setbox0=\hbox{$\textstyle\rm Q$}\hbox{\raise
0.15\ht0\hbox to0pt{\kern0.4\wd0\vrule height0.8\ht0\hss}\box0}}
{\setbox0=\hbox{$\scriptstyle\rm Q$}\hbox{\raise
0.15\ht0\hbox to0pt{\kern0.4\wd0\vrule height0.7\ht0\hss}\box0}}
{\setbox0=\hbox{$\scriptscriptstyle\rm Q$}\hbox{\raise
0.15\ht0\hbox to0pt{\kern0.4\wd0\vrule height0.7\ht0\hss}\box0}}}}

\def\bbbz{{\mathchoice {\hbox{$\sf\textstyle Z\kern-0.4em Z$}}
{\hbox{$\sf\textstyle Z\kern-0.4em Z$}}
{\hbox{$\sf\scriptstyle Z\kern-0.3em Z$}}
{\hbox{$\sf\scriptscriptstyle Z\kern-0.2em Z$}}}}

\def\bbbc{{\mathchoice {\setbox0=\hbox{$\displaystyle\rm C$}\hbox{\hbox
to0pt{\kern0.4\wd0\vrule height0.9\ht0\hss}\box0}}
{\setbox0=\hbox{$\textstyle\rm C$}\hbox{\hbox
to0pt{\kern0.4\wd0\vrule height0.9\ht0\hss}\box0}}
{\setbox0=\hbox{$\scriptstyle\rm C$}\hbox{\hbox
to0pt{\kern0.4\wd0\vrule height0.9\ht0\hss}\box0}}
{\setbox0=\hbox{$\scriptscriptstyle\rm C$}\hbox{\hbox
to0pt{\kern0.4\wd0\vrule height0.9\ht0\hss}\box0}}}}

\title{Index theory for actions of compact Lie groups on $C^*$-algebras}
\author{Charlotte Wahl}


\begin{document}
\begin{abstract}
We study the index theory for actions of compact Lie groups on $C^*$-algebras with an emphasis on principal actions. Given an invariant semifinite faithful trace on the $C^*$-algebra we obtain semifinite spectral triples. For circle actions we consider the relation to the dual Pimsner-Voiculescu sequence. On the way we show that the notions ``saturated'' and ``principal'' are equivalent for actions by compact Lie groups. 
\end{abstract}

\maketitle

\section{Introduction}

$C^*$-algebras can be considered as a noncommutative generalization of locally compact spaces. Let $G$ be a compact group.
A $G$-$C^*$-algebra is a $C^*$-algebra $\A$ endowed with an action $\alpha$ of $G$, such that the homomorphism $\A \to C(G,\A),~a \mapsto \alpha(a)$ is continuous. $G$-$C^*$-algebras generalize actions of a compact groups on locally compact spaces. The fixed point algebra $\A^{\alpha}$ plays the role of the quotient. The notion of a saturated action was introduced in \cite{ph} as a noncommutative version of a free action. If $\alpha$ is saturated, then $\A^{\alpha}$ and the crossed product $\A \times_{\alpha}G$ are strongly Morita equivalent. 
More recently principal actions for quantum groups on $C^*$-algebras were defined \cite{e}. These generalize principal bundles. We will show that for compact Lie groups the notions saturated and principal agree. We refer to \cite{bhms} for a detailed account on the correspondence between topological and $C^*$-algebraic notions in this context. 

Index theorical concepts for actions of Lie groups have been studied at several places in the literature: A general index theory was outlined in \cite{co}. For ergodic actions a spectral triple has been defined in \cite[\S 3]{ri}. For the torus gauge action on (higher) graph $C^*$-algebras a semifinite spectral triple has been constructed and an index theorem has been proven recently in \cite{pre}\cite{prs}.  The aim of this paper is to study the index theory for actions of compact Lie groups in a systematic way and in the case of principal actions to elaborate the analogy with the index theory of vertical elliptic operators on a principal bundle. Our results provide the missing link between the existing constructions and also a connection to the index theory for pseudodifferential operators over $C^*$-algebras in general. Furthermore we obtain a new class of examples of semifinite spectral triples.

We describe now our constructions and their relation to those mentioned before in more detail:

In \cite{co} a pseudodifferential calculus for an $\bbbr^n$-action on a $C^*$-algebra was defined in a very concrete way, an Atiyah-Singer type exact sequence was constructed and an index theorem was proven. It was remarked that a similar theory exists for general locally compact Lie groups. Our approach here is different: We will show that for compact Lie groups the algebra of invariant pseudodifferential operators on $\C(G,\A)$ yields such an Atiyah-Singer type exact sequence (\S \ref{atsinseq}). As in \cite{co} the role of the algebra of compact operators is played by $\A \times_{\alpha} G$ and hence the index is an element in $\Ki_*(\A \times_{\alpha} G)$. An index theorem for the pairing of the index with an invariant trace on $\A$ follows from the index theorem proven in \cite{wa}. 

We will also define and study appropriate analogues of Sobolev spaces, which will be Hilbert $\A^{\alpha}$-modules. If $\A=C(X)$ for a compact space $X$, then these correspond to vertical Sobolev spaces on the fibration $X \to X/G$. Our calculations will show, as has been noted many times before in related situations, that the crossed product is in general better behaved than the quotient from an index theoretic point of view. For saturated actions the pseudodifferential calculus considered here is analogous to the pseudodifferential calculus of vertical invariant classical pseudodifferential operators on a principal bundle (\S \ref{satact}). 

We will particularly focus on Dirac operators: An invariant Dirac operator $\dira$ on $G$ induces an elliptic element $\Dir$ in our calculus and a class $[\Dir]$ in $\KKi_*(\A,\A \times_{\alpha} G)$ (\S \ref{indDirtheor}). We prove an index theorem for the pairing of $\Ki_*(\A)$ with this class. If $\A$ is endowed with a suitable invariant trace, then we associate to any invariant Dirac operator on the group a semifinite spectral triple for a matrix algebra over the enveloping von Neumann algebra of the crossed product (\S \ref{spectrip}). 

In \cite{pre}\cite{prs} the authors associate to any invariant Dirac operator on a $k$-torus $T^k$ a selfadjoint operator with compact resolvents on a Hilbert $\A^{\alpha}$-module, where $\A$ is a $k$-graph algebra and $\alpha$ the gauge action. They show that their construction yields a class in $\KKi_*(\A,\A^{\alpha})$ and -- in presence of an invariant semifinite trace -- define a semifinite spectral triple to which they apply the semifinite local index formula \cite{cpr1}\cite{cpr2}. For ergodic actions of compact Lie groups on general $C^*$-algebras an analogous construction has been indicated before in \cite{ri} yielding a (ordinary) spectral triple. We generalize the construction to saturated actions and show that the $\KKi$-theory class it is related to the class $[\Dir]$ from above via Morita equivalence (\S \ref{satact}). The geometric interpretation of the index theorem in this situation is that of a family index theorem for the invariant Dirac operator on a principal bundle with fiber $G$ and noncommutative base and total space. The von Neumann algebra defined in \cite{prs} was suspected by the authors to be isomorphic to the (von Neumann algebraic) crossed product of the $k$-graph $C^*$-algebra with $T^k$. Our approach shows that this is indeed the case if the action is saturated.

One motivation for index theories lies in the construction and calculation of invariants. Invariants of graph $C^*$-algebras have recently found interest in connection with Mumford curves, see \cite[\S 1.17]{clm}. The pairing of $\Ki_*(\A)$ with $[\Dir]$ yields invariants of the action $\alpha$. If $G=T$, then these are well-known: We will identify two of the six maps of the dual Pimsner-Voiculescu sequence with the Kasparov product with $[\Dir]$ (\S \ref{pimsvoi}). For graph $C^*$-algebras, for example, the dual Pimsner-Voiculescu sequence is well-understood and a useful tool in the study of their $\Ki$-theory \cite[Ch. 7]{r}.

The most popular example of a saturated action is probably the $T^2$-action on the noncommutative two-torus. Other examples arise from the $T$-action on graph $C^*$-algebras and from $C^*$-algebraic quantum principal bundles (\S \ref{examples}). Though in the latter case the fiber is a Hopf-algebra in general, the structure group of the various quantum Hopf bundles is the circle.

\section{Hilbert $C^*$-modules and $\KKi$-theory}

In this section we review some notions related to Hilbert $C^*$-modules and $\KKi$-theory needed in the sequel. We refer to \cite{bl} for more details.

Let $\A$ be a $C^*$-algebra with norm $\|~\|$.

An $\A$-valued scalar product on a right $\A$-module $H$ is a pairing $\langle~,~\rangle: H\times H \to \A$ that is linear in the second variable and such that for $v,w \in H$ and $a \in \A$
\begin{itemize}
\item $\langle v,wa\rangle=\langle v,w\rangle a$,
\item $\langle v,w\rangle=\langle w,v\rangle ^*$,
\item $\langle v,v\rangle\ge 0$,
\item $\langle v,v \rangle=0$ only if $v=0$.
\end{itemize}

The induced norm on $H$ is given by $\|v\|_H^2=\|\langle v,v \rangle \|$.

A Hilbert $\A$-module is a right $\A$-module $H$ endowed with an $\A$-valued scalar product such that $H$ is complete with respect to the induced norm. 

An elementary example of a Hilbert $\A$-module, which will however play a role in the following, is the following: Let $E$ be a finite dimensional complex vector space with hermitian product $\langle~,~\rangle_E$. Then $E \ten \A$ endowed with the $\A$-valued scalar product 
$$\langle v \ten a,w \ten b\rangle= \langle v,w\rangle_E a^*b$$
is a Hilbert $\A$-module.

The $C^*$-algebra of bounded adjointable ($\A$-linear) operators on a Hilbert $\A$-module $H$ is denoted  by $B(H)$. For $v,w \in H$ we define $B(H) \ni \Theta_{v,w}: x\mapsto v \langle w,x\rangle$. The ideal of compact operators in $B(H)$ is per definition the sub-$C^*$-algebra of $B(H)$ generated by the operators $\Theta_{v,w},~v,w \in H$. An operator $F \in B(H)$ is called Fredholm if $F$ is invertible in $B(H)/K(H)$. If $H$ is countably generated, then there is a well-defined index $\Ind(F) \in K_0(\A)$ of a Fredholm operator $F$. 

In an elementary way an ideal a closed ideal $I$ in a $C^*$-algebra $\B$ defines a ``generalized index theory'':  A Fredholm operator is an element $F \in \B$ whose class is invertible in $\B/I$ and the index is the image of the class $[F] \in \Ki_1(\B/I) \to \Ki_0(I)$. The index theory constructed in \S \ref{indseq} will be such a generalized index theory. 

An unbounded densely defined operator $D$ on $H$ is called regular if $(1+D^*D)^{-1}$ has dense range. If $D$ is selfadjoint and regular, then $f(D) \in B(H)$ is well-defined for $f \in C(\bbbr)$.

We assume now that $\A,\B$ are $\sigma$-unital $C^*$-algebras. 

\begin{ddd}
An odd unbounded Kasparov $(\A,\B)$-module is a triple $(\rho,D,H)$, where $H$ is a countably generated Hilbert $\B$-module, $\rho:\A \to B(H)$ is a $C^*$-homomorphism and $D$ is a selfadjoint regular operator on $H$, is called an unbounded Kasparov $(\A,\B)$-module if for all $a \in \A$
$$\rho(a)(1+D^2)^{-1} \in K(H)$$ 
and such that there is a dense subset $A \subset \A$ with $$[D, \rho(a)]$$ densely defined and bounded for all $a \in \A$. 

An even unbounded Kasparov $(\A,\B)$-module is a triple $(\rho,D,H)$ as above such that in addition $H$ is $\bbbz/2$-graded, $D$ is odd and $\rho(\A)$ even. 
\end{ddd}

An even resp. odd unbounded Kasparov $(\A,\B)$-module $(\rho,D,H)$ defines a class $[(\rho,D,H)]$ in $\KKi_0(\A,\B)$ resp. $\KKi_1(\A,\B)$. Usually we will suppress $\rho$ and $H$ from the notation.
 
There is a pairing $$\ten_{\A}:\Ki_i(\A) \times \KKi_j(\A,\B) \to \Ki_{i+j}(\B) \ ,$$ where $i,j \in \bbbz/2$. 

We will only consider the case $i=j$. Then the pairing generalizes the following index theoretic pairings by using $\Ki_0(\bbbc)\cong \bbbz$:

Let $P$ be a symmetric elliptic pseudodifferential operator of order one acting on the sections of a hermitian vector bundle $E$ on a closed Riemannian manifold $M$. Then $(m,P,L^2(M,E))$ is an odd unbounded Kasparov $(C(M),\bbbc)$-module, where $m(f)$ is the multiplication operator associated to $f \in C(M)$. In the following $m$ is suppressed in the notation. 

For a unitary $u \in \C(M,M_n(\bbbc))$ 
$$[u]\ten_{C(M)} [P]= \spfl(((1-t)P+ t u P u^{-1})_{t \in [0,1]}) \in \bbbz \ ,$$
where $\spfl$ denotes the spectral flow and $P$ is understood as acting on $L^2(M,E) \ten \bbbc^n$.

If $E$ is $\bbbz/2$-graded and $P$ is odd, then $(M,P,L^2(M,E))$ is an even unbounded Kasparov $(C(M),\bbbc)$-module. 
For a projection $p \in \C(M,M_n(\bbbc))$
$$[p]\ten_{C(M)} [P]=\Ind(p(P^+(1+P^2)^{-1/2})p \in \bbbz ,$$
where $p(P^+(1+P^2)^{-1/2})p$ is understood as an operator from $pL^2(M,E^+)^n$ to $pL^2(M,E^-)^n)$.

\section{Actions by compact Lie groups on $C^*$-algebras}
\label{indseq}

In this section we recall some facts about actions of compact groups on $C^*$-algebras, which will be needed in the following.

We say that an action $\alpha$ of a compact group $G$ on a Fr\'echet space $V$ is continuous, if $$V \mapsto C(G,V), v \mapsto \alpha(v)$$ is continuous. If not specified we will assume an action to be continuous. The space of fixed points is denoted by $V^{\alpha}$.

If $G$ is endowed with an invariant measure with unit volume, then taking the mean
$$\Phi:V \to V^{\alpha},~v\mapsto \int_G \alpha_g(v)~dg$$ is a continuous projection. In particular if $S \subset V$ is a dense subset with $\Phi(S) \subset S$, then $S^{\alpha}\subset V^{\alpha}$ is dense. We will tacitly make use of this fact. We note that if $V$ is a $C^*$-algebra, then $\Phi$ is an unconditional expectation. However we will not need this fact.

We denote by $\R$ the right regular representation defined for $f \in C(G,V)$ and $g \in G$ by $(\R_gf)(h)=f(hg)$. The diagonal action on $C(G,V)$ of $\alpha$ and $R$ is denoted by $\R^{\alpha}$ and is given by $(\R^{\alpha}_g f)(h)=\alpha_g f(hg)$.

Now let $\A$ be a $C^*$-algebra endowed with an action $\alpha$ of a compact Lie group $G$. For simplicity we assume that $G$ is connected.

Endow $G$ with an invariant Riemannian metric with unit volume.

We refer to \cite[\S 7]{ped} for proofs about the properties of crossed products, which we discuss in the following. Since we deal with amenable groups, there is no need to distinguish between crossed product and reduced crossed product.

The convolution product on $C(G,\A)$ is given by $$(f_1 \st f_2)(g):=\int_G f_1(h)\alpha_hf_2(h^{-1}g)~dh \ .$$ 
For clarity we write sometimes $C(G,\A)^{\st}$ for the space $C(G,\A)$ endowed with the convolution product. 

Let $\pi:\A \to B(H)$ be a faithful representation. There is an induced faithful representation $$\tilde \pi:C(G,\A)^{\st} \to B(L^2(G)\ten H) \ ,$$
$$(\tilde \pi(f_1)f_2)(g):=\int_G \pi(\alpha_g^{-1}(f_1(h))f_2(h^{-1}g)~dh \ .$$ 

The algebra $\A \times_{\alpha}G$ is defined as the completion of $C(G,\A)^{\st}$ with respect to the norm of $B(L^2(G) \ten H)$. The homomorphisms $\A \to B(L^2(G)\ten H), a \mapsto \pi(\alpha(a))$ and $G \to B(L^2(G)\ten H), ~g \mapsto R_g$ define inclusions of $\A$ and $G$ into the multiplier algebra $M(\A \times_{\alpha}G)$. 

There is a $C^*$-homomorphism $C:\A \times_{\alpha} G \to B(L^2(G,\A))$ defined for $f_1\in C(G,\A)^{\st}$ and $f_2 \in C(G,\A) \subset L^2(G,\A)$ by
\begin{align}
\label{actL2}
C(f_1) (f_2)(g) &= \int_G \alpha_g^{-1}(f_1(h)) f_2(h^{-1}g)~dh\\
\nonumber &= \int_G \alpha_g^{-1}(f_1(gh^{-1})) f_2(h)~dh \ .
\end{align}

The aim of the following considerations is to define and study an action on $B(L^2(G,\A))$ induced by $\alpha$:

The diagonal action $\R^{\alpha}$ extends to an action on $L^2(G,\A)$ as a Banach space since for $f_1,f_2 \in C(G,\A)$
\begin{align}
\label{L2scal}
\langle \R^{\alpha}_g f_1,\R^{\alpha}_g f_2\rangle_{L^2} &=\int_G \alpha_g (f_1(hg)^* f_2(hg))~dh \\
\nonumber &= \alpha_g \langle f_1,f_2 \rangle \ ,
\end{align}
hence $\|\R^{\alpha}_g(f)\|=\|f\|$.

In general $\R^{\alpha}_g$ is not a Hilbert $\A$-module morphism on $L^2(G,\A)$. However for any $g \in G$ the map $$\Alpha_g:B(L^2(G,\A)) \to B(L^2(G,\A)),~T \mapsto \R^{\alpha}_g T\R^{\alpha}_{g^{-1}}$$ is a norm-preserving $C^*$-homomorphism, which depends continuously on $g$ with respect to the strong operator topology on $B(L^2(G,\A))$. In particular taking the mean $M_{\Alpha}: B(L^2(G,\A)) \to B(L^2(G,\A))^{\Alpha}$ is continuous with respect to the norm  topology.

Furthermore $\Alpha$ restricts to an action on $K(L^2(G,\A))$ with respect to the norm topology.
  
If $k \in C(G\times G,\A)$ is the integral kernel of an operator $K \in K(L^2(G,\A))^{\Alpha}$, then 
\begin{align}
\label{invint}
k(g,h)=\alpha_{g^{-1}}k(e,hg^{-1}) \ .
\end{align}

Let $f \in C(G,\A)^{\st}$. Since the previous equation holds for 
\begin{align}
\label{intop}
k(g,h)&:=\alpha_g^{-1}(f(gh^{-1})) \ ,
\end{align}
we have that $C(f) \in K(L^2(G,\A))^{\Alpha}$. 

Hence the image of $\A \times_{\alpha} G$ under $C$ is contained in $K(L^2(G,\A))^{\Alpha}$. It is well-known that $$C:\A \times_{\alpha} G \to K(L^2(G,\A))^{\Alpha}$$
is an isomorphism.

In general we will identify $\A \times_{\alpha} G$ with $K(L^2(G,\A))^{\Alpha}$.

The first row of the following commutative diagram defines the index: 
$$\xymatrix{
0 \ar[r]& \A \times_{\alpha} G \ar[r]\ar[d]& B(L^2(G,\A))^{\Alpha} \ar[r]\ar[d]& B(L^2(G,\A))^{\Alpha}/\A \times_{\alpha}G \ar[r]\ar[d]& 0  \\
0 \ar[r]& K(L^2(G,\A)) \ar[r] & B(L^2(G,\A)) \ar[r]& B(L^2(G,\A))/K(L^2(G,\A)) \ar[r]& 0 }$$
commutes.

If $F,G \in B(L^2(G,\A))^{\Alpha}$ are such that $FG-1,~ GF-1$ are compact, then $[F] \in \Ki_1(B(L^2(G,\A))^{\Alpha}/\A \times_{\alpha} G)$ and $\Ind_{\alpha}(F) \in \Ki_0(\A \times_{\alpha} G)$ is defined as the image of $[F]$ under the connecting map $$\Ki_1(B(L^2(G,\A))^{\Alpha}/\A \times_{\alpha} G) \to \Ki_0(\A \times_{\alpha} G) \ .$$ From the diagram and by identifying $\Ki_0(K(L^2(G,\A))) \cong \Ki_0(\A)$ we get that $$\iota_*\Ind_{\alpha}(F)=\Ind(F) \ .$$

\section{Invariant pseudodifferential calculus}
\label{atsinseq}

In this section we study invariant pseudodifferential operators and construct an Atiyah--Singer type exact sequence. For notational simplicity we restrict to the scalar case. We begin by studying subspaces of invariant $\A$-valued functions of the Sobolev spaces $H^k(G,\A)$. These subspaces are Hilbert $\A^{\alpha}$-modules, whereas the Sobolev spaces are Hilbert $\A$-modules. 

An $\A^{\alpha}$-valued scalar product on $\A$ is given by 
$$\langle a,b\rangle_{\A^{\alpha}}:=\langle \alpha^{-1}(a),\alpha^{-1}(b)\rangle_{L^2}= \langle \alpha(a),\alpha(b)\rangle_{L^2}= \int_G \alpha_g(a^*b)~dg \ .$$
We write $H(\A)$ for the Hilbert $\A^{\alpha}$-module completion.

The inclusion $\A \to L^2(G,\A),~ a \mapsto \alpha^{-1}(a)$ extends to an isometric isomorphism 
$$\U(\alpha):H(\A) \cong L^2(G,\A)^{\R^{\alpha}}$$
between the Hilbert $\A^{\alpha}$-modules. 

In particular it follows that left multiplication of $\A$ on $\A$ induces a continuous map $$\A \to B(H(\A))$$ 
and that the inclusion $$\A \to H(\A)$$ is continuous.

Furthermore as in eq. \ref{L2scal} one can show that for $g \in G$ the map $\alpha_g:\A \to \A$ extends to a unitary on $H(\A)$ depending in a strongly continuous way on $g$. The extension will be denoted by $\alpha$ as well.

The algebra $\Ai:=\{a \in \A~|~ \alpha(a) \in \C(G,\A)\}$ is dense and closed under holomorphic functional calculus in $\A$. It inherites the structure of a locally $m$-convex Fr\'echet algebra from $\C(G,\A)$, since $\U(\alpha)$ restricts to a bijection between $\Ai$ and  $\C(G,\A)^{\R^{\alpha}}$.

Analogously we can define 
$$\A_i=\{a \in \A~|~ \alpha(a) \in C^i(G,\A)\} \ .$$

Let $\Delta$ be the scalar Laplacian on $G$.
Let $s \in \bbbr$ and let $H^s(G,\A)$ be the Sobolev space of order $s$, which we define here as the Hilbert $\A$-module completion of $\C(G,\A)$ with respect to the $\A$-valued scalar product
$$\langle f_1,f_2\rangle_{H^s} = \langle (1+\Delta)^{s/2}f_1,(1+\Delta)^{s/2}f_2\rangle_{L^2} \ .$$ 

The theory of Sobolev spaces and pseudodifferential operators over unital $C^*$-algebras was introduced in \cite{mf}. We consider $H^s(G,\A)$ as a sub Banach space of $H^s(G,\tilde{\A})$ in order to make use of the results of \cite{mf}.

Since the operators $(1+\Delta)^{s/2}$ are invariant, it is straightforward to verify that the action $\R^{\tilde \alpha}$ on $\C(G,\A)$ extends to a continuous action on $H^s(G,\A)$. We define the Hilbert $\A^{\alpha}$-module
$$H^s(\A)=\{x \in H(\A)~|~ \alpha(x) \in H^s(G,\A)\} \ ,$$ 
which is isometrically isomorphic to $H^s(G,\A)^{\R^{\tilde \alpha}}$ via $\U(\alpha)$. The algebra $\Ai$ is dense in $H^s(\A)$ for all $s \in \bbbr$.

Sobolev's Lemma implies that for $s>\frac{\dim G}{2}+i$ there is a continuous embedding $$H^s(\A) \incl \A_i$$ and that there is an isomorphism of Fr\'echet spaces 
$$\bigcap_{s \in \bbbr} H^s(\A) \cong \Ai \ .$$

We say that a regular operator $T$ on $L^2(G,\A)$ is invariant if its bounded transform $T(1+T^*T)^{-1/2}$ is invariant.

We get from the previous arguments:

\begin{prop}
\label{isoop}
For $r,s \in \bbbr$ there is a continuous homomorphism 
$$B(H^r(G,\A),H^s(G,\A))^{\Alpha} \to B(H^r(\A),H^s(\A)) \ ,$$
$$T \mapsto T_{\alpha}:=\U(\alpha)^*T\U(\alpha) \ .$$

In particular we get an induced homomorphism $$C_{\alpha}:\A\times_{\alpha} G \to B(H(\A)),~f \mapsto C_{\alpha}(f):=C(f)_{\alpha} \ .$$

Furthermore there is an induced map sending invariant regular operators on $L^2(G,\A)$ to regular operators on $H(\A)$.
\end{prop} 

\begin{proof}
It remains to show the last assertion.
Let $T$ be an invariant regular operator on $L^2(G,\A)$. We denote by $F$ its bounded transform. Since $\Ran(1-F^*F)^{1/2}$ is dense in $L^2(G,\A)$, we have that $\Ran (1-F_{\alpha}^*F_{\alpha})^{1/2}$ is dense in $H(\A)$. By \cite[Th. 10.4]{l} this implies that $F_{\alpha}$ is the bounded transform of a unique regular operator $T_{\alpha}$ on $H(\A)$ given by $\dom T_{\alpha}=\Ran (1-F_{\alpha}^*F_{\alpha})^{1/2}$ and $T_{\alpha}x:=F_{\alpha}y$ if $x=(1-F_{\alpha}^*F_{\alpha})^{1/2}y$.
\end{proof}

For $f \in C(G,\A), a \in \A \subset H(\A)$ it holds that
$$C_{\alpha}(f)(a)=\int_G f(g) \alpha_g(a)~dg \in \A\subset H(\A)\ .$$ 

Note that for $a,b \in \A$ 
$$C_{\alpha}(a\alpha(b^*))=\Theta_{a,b}:z \mapsto a\langle b,z\rangle_{\A^{\alpha}} \ .$$ Hence $K(H(\A))$ is contained in the image of $C_{\alpha}$. We will see in \S \ref{satact} that the action $\alpha$ is principal if and only if $C_{\alpha}$ is an isomorphism onto $K(H(\A))$. In general $C_{\alpha}$ need not be injective and its image not contained in $K(H(\A))$. 

For the remainder of this section assume that $\A$ is unital. Then we have an analogue of the Atiyah exact sequence for pseudodifferential operators: We denote the closure in $B(L^2(G,\A))$ of the algebra of pseudodifferential operators of nonnegative order on $\C(G,\A)$ by $\ov{\Psi}(G,\A)$. 

Let $SG$ denote the sphere bundle of $G$ with the induced $G$-action.

\begin{prop}
There is a commuting diagram 
$$\xymatrix{
0 \ar[r]& \A \times_{\alpha} G \ar[r]\ar[d]& \ov{\Psi}(G,\A)^{\Alpha} \ar[r]\ar[d]& C(SG,\A)^{\R^{\alpha}} \ar[r]\ar[d]& 0  \\
0 \ar[r]& K(L^2(G,\A)) \ar[r]& \ov{\Psi}(G,\A) \ar[r]^{\sigma}& C(SG,\A) \ar[r]& 0 }$$
with exact rows.
\end{prop}

\begin{proof}
We only have to show that $\sigma:\ov{\Psi}(G,\A)^{\Alpha} \to C(SG,\A)^{\R^{\alpha}}$ is surjective. For $f \in C(SG,\A)^{\R^{\alpha}}$ let $P(f) \in \ov{\Psi}(G,\A)$ be such that $\sigma(P(f))=f$. It follows from \cite[Lemma 6.1 ]{tr}, that the action $\Alpha$ is well-defined and continuous on $\ov{\Psi}(G,\A)$. Hence the mean $\Phi(P(f))$ is an element in $\ov{\Psi}(G,\A)^{\Alpha}$ with $\sigma(\Phi(P(f)))=f$.
\end{proof}

There is a $G$-equivariant isomorphism of sphere bundles $G \times S^{\dim G-1} \to SG$, where the $G$-action on $S^{\dim G-1}$ is trivial. Hence $$C(SG,\A)^{\R^{\alpha}} \cong C(S^{\dim G-1},\A)$$ via restriction to the fiber over $e \in G$. 
The rows are subsequences of the rows in diagram at the end of \S \ref{indseq}. Hence if $Q$ is an invariant pseudodifferential operator, then $\Ind_{\alpha}(Q)$ does only depend on the class $[\sigma(Q)] \in \Ki_1(C(SG,\A)^{\R^{\alpha}})$ of the symbol of $Q$. The sequence above corresponds to the sequence defined in \cite{co} for $G=\bbbr^n$.

The index theory for invariant pseudodifferential operators over $C^*$-algebras has been studied in \cite{tr}. The index defined there is an element in $K_*^G(\A)$. It agrees with $\Ind_{\alpha}$ via the isomorphism $K_*^G(\A) \cong K_*(\A\times_{\alpha} G)$. Also the powerful machinery of equivariant $KK$-theory can be applied to the present situation. One can apply the index theorem in \cite[\S 4.2]{wa} in order to obtain \cite[Th. 10]{co} for compact Lie groups.

\section{An index formula for Dirac operators}
\label{indDirtheor}

In the following we consider the index theory for Dirac operators in more detail. We associate a class in $KK_*(\A,\A \times_{\alpha} G)$ to an invariant Dirac operator and give the index theorem for the pairing with $K_*(\A)$. Here $\A$ is allowed to be nonunital. For the index theorem we will assume the existence of an approximative unit with certain properties.

We assume that $G$ is endowed with the trivial spin structure. Write $\gl:=T_eG$. We identify $TG$ with $G \times \gl$. 

Let $E$ be a hermitian finite dimensional vector space that is a Clifford module of the Clifford algebra $Cl(\gl)$ and let $G$ act trivially on it. Then $E \times G$ is a Clifford bundle on $G$. Endow it with an invariant Clifford connection and let $\dira_E$ be the associated Dirac operator. If $\dim G$ is even, then $E$ is $\bbbz/2$-graded and $\dira_E$ is odd. 

Then $$\dira_E=\sum_{i=1}^{\dim G} c(X_i)X_i +\omega \ ,$$ where $\{X_i,~i=1, \dots, \dim G\}$ is an orthonormal basis of $\gl$ and $\omega \in \End E$.

We denote the operator $\dira_E$ acting on the Hilbert $\A$-module $L^2(G, E \ten \A)$ by $\Dir_E$ in the following. Taking closures is understood where necessary.

The operator $\Dir_E$ acts also as an operator on $K(L^2(G,E \ten \A))$ by $$(\Dir_E +i)^{-1}K(L^2(G,E \ten \A)) \to K(L^2(G,E \ten \A)),~ K \mapsto \Dir_E K \ .$$
We fix a unit vector $v \in E$ (which we consider as an element of the dual space of $E$) and identify $K(L^2(G,\A)) \ten E$ with $K(L^2(G,\A)) \ten (E \ten v) \subset K(L^2(G,E \ten \A))$. The following constructions do not depend on the choice of $v$. The operator $\Dir_E$ defines a regular operator on the Hilbert $K(L^2(G,\A))$-module $K(L^2(G,\A)) \ten E$ and restricts to a regular operator on the Hilbert $\A \times_{\alpha} G$-module $\A \times_{\alpha} G \ten E$. A concrete formula for the action of the vectors $X_i$ can be obtained from eq. \ref{actL2}: If $f_1,f_2 \in \C(G,\A)$ fulfill $X_iC(f_1)=C(f_2)$, then (with $X_i$ acting on the variable $g$)
$$\alpha_g^{-1}f_2(gh^{-1})=X_i(\alpha_g^{-1}f_1(gh^{-1}))=\alpha_h^{-1}X_i(\alpha_{gh^{-1}}^{-1}f_1(gh^{-1})) \ ,$$
thus, by setting $h=e$, we get that 
\begin{align}
\label{vecfie}
f_2=\alpha(X_i (\alpha^{-1}f_1)) \ .
\end{align} 

For $a \in \Ai$ we have that 
\begin{align}
\label{comm}
[\Dir_E,\alpha^{-1}(a)]\in \C(G,\A)^{\R^{\alpha}} \ten \End E \ .
\end{align} 

Since furthermore $\alpha^{-1}(a)(\Dir_E^2+1)^{-1} \in \A \times_{\alpha} G\ten \End E$ for all $a \in \A$, the triple $(\alpha^{-1},\Dir_E,\A \times_{\alpha} G\ten E)$ is an unbounded Kasparov $(\A,\A \times_{\alpha} G)$-module, whose parity equals the parity of $\dim G$. The induced class is denoted by $[\Dir_E] \in \KKi_{\dim G}(\A,\A \times_{\alpha} G)$. By homotopy invariance this class does not depend on $\omega$.

In the following we deduce an index theorem for the pairing of $\Ki_*(\A)$ with $[\Dir_E]$ from the index theorem proven in \cite[\S 4.2]{wa} and its even counterpart. The main point here is that we allow for non-unital algebras, motivated by the situation considered in \cite{pre}.

We assume that there is an approximate unit $\{e_i\}_{i \in \bbbn}\subset \Ai^{\alpha}$ of $\A$ such that $e_ie_{i+1}=e_i$. See \cite[\S 2.2]{pre} and references therein for the relevance of this type of approximate unit in the context of spectral triples.

Define $\A_i$  as the closure of the algebra $e_i\A e_i$ in $\A$.  Let $\A_c=\cup_{i \in \bbbn} \A_i$ endowed with the direct limit topology. 

Let $\tau$ be an invariant trace on $\A$ that restricts to a continuous trace on $\A_i$ for any $i \in \bbbn$. 

We denote the induced trace on $\A_c \ten \End E$, which is defined by tensoring with the canonical trace on $\End E$, by $\tau$ as well. 

We get an induced homomorphism 
$$\Tr_{\tau}: \Ki_0(\A \times_{\alpha} G) \to \Ki_0(K(L^2(G,\A))) \cong \Ki_0(\A) \cong \lim_{i \to \infty}\Ki_0(\A_i) \stackrel{\tau}{\lr} \bbbr \ .$$

In the following proposition we identify $\Ki_*(\A)$ with $\Ker(\Ki_0(\tilde \A) \to \Ki_*(\bbbc))$.

\begin{prop} 
\label{indtheo}
\begin{enumerate} 
\item Assume that $\dim G$ is even. 
For projections $p,q \in M_n(\tilde \A_i)$ with $[p]-[q] \in \Ki_0(\A)$ and such that $\alpha(p),\alpha(q) \in \C(G,M_n(\tilde \A_i))$ we have that
$$\Tr_{\tau}(([p]-[q])\ten_{\A}[\Dir_E])=C\tau(p[\Dir_{\alpha},p]^{\dim G})-\tau(q[\Dir_{\alpha},q]^{\dim G}) \ .$$
\item Assume that $\dim G$ is odd. 
For a unitary $u \in M_n(\tilde \A_i)$ with $\alpha(u) \in \C(G,M_n(\tilde \A_i))$ we have that
$$\Tr_{\tau}([u]\ten_{\A}[\Dir])=C\tau\left(u^{-1}[\Dir_{\alpha},u]([\Dir_{\alpha},u^{-1}][\Dir_{\alpha},u])^{(\dim G-1)/2}\right) \ .$$
\end{enumerate}
Here $C \in \bbbc$ is a constant, which only depends on $\dim G$.
\end{prop}

We give the index formulas in the present form in order to point out the analogy to other index formulas in noncommutative geometry. For calculations one may use the explicit formulas given in \cite{wa}.

\begin{proof} For unital $\A$ and finite trace the proposition follows in a standard way from the results in \cite{wa} by using $[\Dir, f]=c(df)$ for $f \in \C(G,\A)$ and the relation between the trace on the Clifford algebra and the Berezin integral on the exterior algebra.  

From this we deduce the general case:

The action $\alpha$ restricts to an action $\alpha_i$ on $\A_i$. We also have a class $[\Dir_i] \in \KKi_{\dim G}(\A_i,\A_i \times_{\alpha_i} G)$ defined as above from $\dira_E$. The diagram
$$\xymatrix{
\Ki_*(\A_i) \ar[rr]^-{\ten_{\A_i} [\Dir_i]}\ar[d]&& \Ki_{*+\dim G}(\A_i \times_{\alpha_i} G) \ar[r]^-{\tau}\ar[d] & \bbbc \ar[d]^{=}\\
\Ki_*(\A) \ar[rr]^-{\ten_{ \A } [\Dir_E ]}&& \Ki_{*+\dim G}(\A \times_{\alpha} G) \ar[r]^-{\tau} & ~\bbbc \ .} $$
commutes and the vertical maps, which are induced by the inclusion, become isomorphisms in the inductive limit. The trace extends to a trace $\tau_i$ the unitalization $\tilde \A_i$ by setting $\tau_i(1)=\tau(e_{i+1})$. We also have induced action $\tilde \alpha_i$ on  $\tilde \A_i$ and an element $[\tilde \Dir_i] \in \KKi_{\dim G}(\tilde \A_i,\tilde \A_i \times_{\tilde \alpha_i} G)$.

As before there is a commutative diagram
$$\xymatrix{
\Ki_*(\A_i) \ar[rr]^-{\ten_{\A_i} [\Dir_i]}\ar[d]&& \Ki_{*+\dim G}(\A_i \times_{\alpha_i} G) \ar[d]\ar[r]^-{\tau} &\bbbc \ar[d]^=\\
\Ki_*(\tilde \A_i) \ar[rr]^-{\ten_{\tilde \A_i} [ \tilde \Dir_i]} && \Ki_{*+\dim G}(\tilde \A_i\times_{\tilde \alpha_i} G) \ar[r]^-{\tau_i} &\bbbc \ .}$$

The assertion follows now from the unital case.
\end{proof}

\section{Constructing semifinite von Neumann algebras}
\label{semtrac}

We discuss some constructions concerning semifinite von Neumann algebras. See \cite{d} for more information on semifinite von Neumann algebras in general and on Hilbert algebras.

For a von Neumann algebra $\N$ with semifinite trace $\tau$ we denote by $\K(\N)$ the closure of the ideal $$l^1(\N)=\{A \in \N~ |~ \tau(|A|) < \infty\} \ .$$ Since $l^1(\N)$ is dense and closed under holomorphic functional calculus in $\K(\N)$ the trace induces a homorphims $\tau: \Ki_0(\K(\N)) \to \bbbc$.

In order to streamline the argumentation in the following we introduce the following notion:

A pair $(A,\tau)$, where $A$ is a pre-$C^*$-algebra and $\tau$ a trace on $A$, is called Hilbert pair, if $A$ is a Hilbert algebra with respect to the scalar product $(a,b) \mapsto \tau(a^*b)$ and if multiplication from the left induces a continuous map $\pi: A \to B(L_{\tau}^2(A))$. Here $L_{\tau}^2(A)$ denotes the Hilbert space completion of $A$. 

Denote by $\ov{A}$ the completion of $A$, which is a $C^*$-algebra.
If $(A,\tau)$ is a Hilbert pair, then $\pi$ is faithful and extends to a faithful representation of the multiplier algebra $M(\ov{A})$. Furthermore the trace $\tau$ extends to a semifinite trace on the von Neumann algebra $A''$ and we have that $\pi(A) \subset l^2(\N)$, $\pi(\ov A) \subset \K(A'')$ and $\pi((M(\ov{A})) \subset A''$ (see \cite[\S 3.12]{ped}). In particular there is a homomorphism
$$\Ki_0(\ov A) \to \Ki_0(\K(A'')) \stackrel{\tau}{\lr} \bbbc \ .$$

In general we will drop $\pi$ from the notation and identify $\ov A$ with its image in $B(L_{\tau}^2(A))$.

We also need tensor products:

Assume that $(A_1,\tau_1)$, $(A_2,\tau_2)$ are Hilbert pairs. We consider the algebraic tensor product $A_1 \odot A_2$ as a subalgebra of the spatial tensor product $\ov{A}_1 \ten \ov{A}_2$. (In our case $\ov{A}_1$ will always be nuclear, so that any tensor product works.) Then $(A_1 \odot A_2,\tau:= \tau_1 \ten \tau_2)$ is a Hilbert pair: The above construction induces a Hilbert structure on $A_1 \ten A_2$ and $L_{\tau}^2(A \ten B) =L^2_{\tau}(A) \ten L^2_{\tau}(B)$. This implies that we get a continuous map $A \ten B \to B(L_{\tau}^2(A\ten B))$.

\section{Semifinite spectral triples from Dirac operators}
\label{spectrip}

Semifinite spectral triples, also called unbounded Breuer-Fredholm modules, were first introduced in \cite{cp} as a generalization of spectral triples. 

\begin{ddd} Let $\N$ be a semifinite von Neumann algebra endowed with a semifinite normal faithful trace, $A$ an involutive Banach algebra with a continuous morphism of involutive Banach algebras $A \to \N$ (which will be suppressed in the notation) and $D$ a selfadjoint operators affiliated to $\N$. The triple $(A,D,\N)$ is called an odd $p$-summable semifinite spectral triple if there is a dense involutive subalgebra $A_c$ of $A$ such that for all $a \in A_c$
$$a(1+D^2)^{-1/2} \in l^p(\N)$$
and
$$[a,D] \in \N \ .$$
The triple $(A,D,\N)$ is called an even semifinite spectral triple if $\N$ is in addition $\bbbz/2$-graded, the operator $D$ is odd and the image of $A$ is in $\N^+$. 
\end{ddd}

An odd resp. even semifinite spectral triple induces a homomorphism from $\Ki_1(A)$ resp. $\Ki_0(A)$ to $\bbbr$ (see \cite[\S 2.4]{bf}). 

In this section we associate a semifinite spectral triple to the Kasparov module $(\alpha^{-1},\Dir_E,\A \times_{\alpha} G\ten E)$ defined in \S \ref{indDirtheor}. 

First we define the semifinite von Neumann algebra.

We assume that the trace $\tau$ is faithful. Then $(\A_c,\tau)$ is a Hilbert pair. Denote by $\Tr$ the canonical trace on the space of trace class operators on $L^2(G,E)$. We write $I:C(G\times G,\End E) \to K(L^2(G,E))$ for the map assigning to an integral kernel the corresponding operator and $\Cg(E)$ for the image of $I$. The pair $(\Cg(E),\Tr)$ is also a Hilbert pair. 

The tensor product construction in the previous section implies that the trace $\Tr \ten \tau$ on the dense subalgebra $\Cg(E) \ten \A \subset K(L^2(G,E \ten \A))$ extends to a semifinite trace $\Tr_{\tau}$ on the enveloping von Neumann algebra $\M$ of $K(L^2(G,E \ten \A))$. Furthermore the tensor product of the map $I$ with the identity on $\A_c$ defines a continuous map $C(G\times G,\End E \ten \A_c)\to l^2(\M)$.

We denote by $\N$ the enveloping von Neumann algebra of $\A\times_{\alpha} G$ and consider $\N\ten \End E$ as a subalgebra of $\M$.
The trace $\Tr_{\tau}$ restricts to a semifinite trace on $\N\ten \End E$ (here we use its invariance). 

It is clear that for $E=\bbbc$ the composition
$$\Ki_0(\A \times_{\alpha} G) \to \Ki_0(\K(\N)) \stackrel{\Tr_{\tau}}{\lr} \bbbr$$ agrees with the map $\Tr_{\tau}$ defined before Prop. \ref{indtheo}, which justifies our notation.

Next we define an operator affiliated to $\N \ten \End E$ from $\Dir_E$.

For an element $\lambda$ in the (discrete) spectrum $\sigma(\Dir_E)$ of $\Dir_E$ let $P(\lambda) \in B(L^2(G,E \ten \A))$ be the projection onto the corresponding eigenspace. Then $e_iP(\lambda) \in \Cg(E \ten \A_c)^{\Alpha} \subset l^2(\N)\ten \End E$. The equality $e_iP(\lambda) e_{i+1}P(\lambda)=e_iP(\lambda)$ implies that $e_iP(\lambda) \in l^1(\N)$. Let $H$ be a Hilbert space on which $\N\ten \End E$ is faithfully represented.
We define $\D_E$ as the closure of the unbounded symmetric operator $\oplus_{\lambda \in \sigma(\Dir_E)} \lambda P(\lambda)$ on $H$. This is a selfadjoint operator affiliated to $\N\ten \End E$.  

Furthermore we have a map $\A \to M(\A \times_{\alpha}G) \to \N$ and by eq. \ref{comm} 
$$[a,\D_E] \in \N\ten \End E$$ 
for all $a \in \Ai$.

For $f\in C_0(\bbbr)$ and $a\in \A$ we have that $af(\D_E) \in \K(\N)\ten \End E$.

If $a \in \A_c$ and $f \in C_0(\bbbr)$ is a positive function such that $f(\Dir_E) \in l^1(L^2(G,E))$, then $af(\D) \in l^1(\N) \ten \End E$.

We summarize:

\begin{prop}
The triple $(\A,\D_E,\N\ten \End E)$ is a semifinite spectral triple. Its degree (even or odd) equals the parity of the dimension of the group $G$. The semifinite spectral triple is $p$-summable for $p>\dim G$.
\end{prop}

The results in \cite{knr} imply that the induced map $\Ki_{\dim G}(\A) \to \bbbr$ agrees with the composition $\Ki_{\dim G}(\A) \stackrel{\ten_{\A}[\Dir_E]}{\lr} \Ki_0(\A \times_{\alpha} G) \stackrel{\Tr_{\tau}}{\lr} \bbbr$.
Using this, one may probably get an alternative proof of Prop. \ref{indtheo} for faithful traces from the local index formula in the semifinite setting \cite{cpr1}\cite{cpr2}.

\section{The dual Pimsner-Voiculescu sequence}
\label{pimsvoi}

This section is devoted to circle actions. We will see how the dual Pimsner-Voiculescu sequence is related to the index theory studied above. We refer to \cite[\S 10.6]{bl} for its definition and proof. At some points we will use different sign conventions.

Let $T=\bbbr/\bbbz$. We take $E=\bbbc$ and write $\Dir=-i \frac{d}{dx}$.

There is a dual action $\hat \alpha$ of $\bbbz$ on $\A \times_{\alpha} T$: Let $S\in B(L^2(T,\A))$ be defined by $$(Sf)(x)=e^{2\pi i x}f(x) \ .$$ Then $B(L^2(T,\A))^{\Alpha} \to B(L^2(G,\A))^{\Alpha},~ A \mapsto SAS^{-1}$ is well defined and  $$C(\hat \alpha(1)f)=SC(f)S^{-1}$$ for $f \in  \A \times_{\alpha} T$.

The dual Pimsner-Voiculescu sequence associated to an action by the circle $T$ on $\A$ is, with $\hat\alpha_*:=\hat \alpha(1)_*$,
$$\xymatrix{
\Ki_0(\A \times_{\alpha} T) \ar[r]^{1-\hat \alpha_*} & \Ki_0(\A \times_{\alpha} T) \ar[r]^{\iota_*} & \Ki_0(\A) \ar[d]^{q^*} \\
\Ki_1(\A)\ar[u]_{q^*}&  \Ki_1(\A \times_{\alpha} T) \ar[l]_{\iota_*} &  \Ki_1(\A \times_{\alpha} T) \ar[l]_{1-\hat \alpha_*} \ .}$$

The definition of the map $q_*$ will be recalled in the proof of the following proposition. The definition of $\iota_*$ will be given in the second proposition.

\begin{prop}
\label{toract}
The map 
$$\ten_{\A} [\Dir]:\Ki_*(\A) \to \Ki_{*+1}(\A \times_{\alpha} T)$$ 
agrees with $q_*$.
\end{prop}

\begin{proof}
We begin by recalling the definition of the map $q_*$: There is a surjection $p:\A \times_{\alpha}\bbbr \to \A \times_{\alpha}T$, which will be given below in detail. Here the crossed product $\A \times_{\alpha}\bbbr$ is defined via the projection $\bbbr \to T$. Composing the Connes-Thom-isomorphism $\Ki_*(\A) \cong \Ki_{*+1}(\A \times_{\alpha}\bbbr)$ with the induced map $p_*$ gives $q_*$.  
    
The map $p$ is constructed using arguments from the proof of \cite[Prop. 10.4.2]{bl}: 

Represent $\A$ faithfully on a Hilbert space $H$. We use the faithful representation $\pi$ of $\A \times_{\alpha} T$ on $L^2(T) \ten H$ given for $h \in C(T,\A)$ by 
$$(\pi(h)f)(t)= \int_{T} \alpha_{-t}(h(r)) f(t-r)~dr \ .$$
A representation $\mu$ of $\A\times_{\alpha}\bbbr$ on the Hilbert $C(T)$-module $C(T) \ten L^2(T) \ten H$ is defined for $h \in C_c(\bbbr,\A)$ by
$$(\mu(h)f)(s,t)= \int_{\bbbr} \alpha_{-t}(h(r)) e^{2 \pi i r s}f(s,t-r)~dr \ .$$ 
Evaluation $\ev: C(T) \to \bbbc$ at the point $0$ induces a homomorphism $\ev_*:B(C(T) \ten L^2(T) \ten H) \to B(L^2(T) \ten H)$, since the Hilbert module tensor product $(C(T) \ten L^2(T) \ten H) \ten_{\ev} \bbbc$ is isomorphic to $L^2(T) \ten H$.
We see that $(\ev_* \circ \mu)(\A\times_{\alpha}\bbbr) = \pi(\A \times_{\alpha}T)$. The map $\ev_* \circ\mu$ defines the map $p$. Note that $p(f)=f \in \A \times_{\alpha}T$ for $f  \in C_0((0,1),\A)\subset \A \times_{\alpha}\bbbr$.

We also need a concrete description of the Connes-Thom-isomorphism. It is given by the Kasparov product with the Thom element in $\KKi_1(\A, \A \times_{\alpha}\bbbr)$, whose definition, due to Fack--Skandalis, we recall now \cite[\S 19.3]{bl}:

The Fourier transform $\Ft: L^2(\bbbr,\A) \to L^2(\bbbr,\A)$ is a unitary. 

Note that for $f_1 \in C_c(\bbbr),~f_2\in C_c(\bbbr,\A)$
$$f_1 \st f_2=\alpha(f_1 \star (\alpha^{-1} f_2)) \ ,$$
where $\star$ denotes the convolution product for the trivial action.

There is a homomorphism $B:C(\bbbr) \to M(\A \times_{\alpha} \bbbr)$ given for $f_1 \in C(\bbbr),~f_2 \in C_c(\bbbr,\A)$ by $B(f_1)f_2=\alpha(\Ft^{-1}(f_1 \Ft(\alpha^{-1}f_2)))$.
 
It extends to a map from unbounded continuous functions to unbounded regular operators on the Hilbert $\A \times_{\alpha}\bbbr$-module $\A \times_{\alpha}\bbbr$.

The Thom element is the class $[B(x)] \in \KKi_1(\A, \A \times_{\alpha}\bbbr)$. For $f \in \C_c(\bbbr,\A)$  
$$B(x)f=\alpha(\Ft^{-1}(x \Ft(\alpha^{-1}f)))=-\alpha(i\frac{d}{dx} \alpha^{-1}f) \ .$$
The assertion follows since for $f \in \C(T,\A)^{\st} \subset \A \times_{\alpha}T$ eq. \ref{vecfie} implies that $$\Dir f=-\alpha(i\frac{d}{dx} \alpha^{-1}f) \ .$$  
\end{proof}

We recall that $\Ki_*(K(L^2(T,\A))) \cong \Ki_*(\A)$ via Morita equivalence.

\begin{prop}
The map $$\iota_*: \Ki_*(\A \times_{\alpha} T) \to \Ki_*(\A)$$ is induced by the inclusion $$K(L^2(T,\A))^{\Alpha} \to K(L^2(T,\A)) \ .$$  
\end{prop}

\begin{proof}
Per definition the map $\iota_*$ is induced from the inclusion $$\iota:\A \times_{\alpha} T \to (\A \times_{\alpha} T) \times_{\hat \alpha} \bbbz,~ f \mapsto f  \delta_0 \ ,$$
by using Takai duality $(\A \times_{\alpha} T) \times_{\hat \alpha} \bbbz \cong K(L^2(T,\A))$ and Morita equivalence. The Takai duality isomorphism assigns to $f \delta_k$ with $f \in \A \times_{\alpha} T$ the operator $C(f)S^k \in K(L^2(T,\A))$. This implies the assertion.
\end{proof}

\section{Saturated actions}
\label{satact}

In this section we specify to saturated actions. We will see that in the commutative case the pseudodifferential calculus defined in \ref{atsinseq} agrees with the calculus of vertical pseudodifferential operators on a principal bundle. Under this analogy $\A$ corresponds to the algebra of continuous functions on the total space and the fixed point algebra $\A^{\alpha}$ to the algebra of continuous functions on the base space. 

Recall from \cite[Cor. 7.1.5]{ph}:

\begin{ddd}
The action $\alpha$ is called saturated if the span $\Fi$ of the functions $a\alpha(b^*),~a,b \in \A$ is dense in $\A\times_{\alpha} G$. 
\end{ddd}

Since $C_{\alpha}(\Fi)$ consists of finite rank operators on $H(\A)$, we have that $C_{\alpha}(\A\times_{\alpha} G) \subset K(H(\A))$ if $\alpha$ is saturated. This shows part of the following proposition. 

\begin{prop}
\label{isocomp}
The action $\alpha$ is saturated if and only if $C_{\alpha}$ maps $\A\times_{\alpha} G$ isomorphically onto $K(H(\A))$.
\end{prop}

\begin{proof}
Assume that $\alpha$ is saturated. 
It remains to show that $C_{\alpha}$ is injective. By (a slight modification due to different conventions of) \cite[Prop. 7.1.3]{ph} there is a $\A\times_{\alpha} G$-valued scalar product $\langle~,~\rangle_{\A\times_{\alpha} G}$ making the left $\A\times_{\alpha} G$-module $H(\A)$ a Hilbert $\A\times_{\alpha} G$-module. Since $\alpha$ is saturated, the scalar product is essential by \cite[Def. 7.1.4]{ph}. Let $0 \neq a \in \A \times_{\alpha} G$. Then there are $v,w \in H(\A)$ such that $$0 \neq a\langle v,w \rangle_{\A\times_{\alpha} G}=\langle C_{\alpha}(a)v,w\rangle_{\A\times_{\alpha} G} \ .$$ 
Hence $C_{\alpha}(a) \neq 0$.

The converse is clear. 
\end{proof}

\begin{prop}
\label{invcomp}
Assume that $\alpha$ is saturated. Then 
$$B(H^r(G,\A),H^s(G,\A))^{\Alpha} \to B(H^r(\A),H^s(\A)),~T \mapsto T_{\alpha}$$
is an isomorphism, which maps compact operators to compact operators.
 \end{prop} 

\begin{proof}
The maps $(1+\Delta)^{s/2}:H^s(G,\A))^{\R^{\alpha}} \to L^2(G,\A)^{\R^{\alpha}}$ and $(1+\Delta_{\alpha})^{s/2}:H^s(\A) \to H(\A)$ are isomorphisms. Hence it is enough to consider the case $s=0$. The map $$K(L^2(G,\A))^{\Alpha} \to K(H(\A)),~T \mapsto T_{\alpha}$$
is an isomorphism by the previous proposition. Let $T \in B(L^2(G,\A))^{\Alpha}$. Since $(1+\Delta)^{-1}$ has dense range, $0 \neq T(1+\Delta)^{-1} \in K(L^2(G,\A))^{\Alpha}$. It follows that $T_{\alpha}(1+\Delta_{\alpha})^{-1} \neq 0$, thus $T_{\alpha} \neq 0$. This shows injectivity.

For surjectivity let $S \in B(H(\A))$ and let $\tilde S_n \in K(L^2(G,\A))^{\Alpha}$ be the preimage of $Se^{-\frac 1n\Delta} \in K(H(\A))$. Then $(\tilde S_n(1+\Delta)^{-1})_{n \in \bbbn}$ is a Cauchy sequence in $K(L^2(G,\A))$. Hence $\tilde S_n$ converges strongly on the range of $(1+\Delta)^{-1}$ to an operator $\tilde S$. Since the sequence $(\tilde S_n)_{n \in \bbbn}$ is uniformly bounded by $\|S\|$ and the range of $(1+\Delta)^{-1}$ is dense, the operator $\tilde S$ is bounded as well. Furthermore the map $T \mapsto T_{\alpha}$ is also continuous with respect to the strong operator topology on both sides. Hence the image of $\tilde S$ is $S$.
\end{proof}

We assume in the following that $\alpha$ is saturated and that $\A$ is separable.

By Morita equivalence $\Ki_*(\A^{\alpha})\cong \Ki_*(K(H(\A))$.
In view of the first proposition we recover that 
$$\Ki_*(\A^{\alpha}) \cong \Ki_*(\A\times_{\alpha} G) \ .$$ By \cite[Lemma 7.1.7]{ph} the map $$\phi:\A^{\alpha} \to \A\times_{\alpha} G,~a \mapsto a \in C(G,\A)^{\st}$$ induces the isomorphism.

The second proposition implies that multiplication of $a \in \Ai$ from the left induces a compact map from $H^r(\A)$ to $H^s(\A)$ for $r>s$ and that for $a \in \A$ and $k<0$ the operator $a Q_{\alpha}$ is compact on $H(\A)$. 

The Dirac operator $\Dir_E$ on $L^2(G,E \ten \A)$ defined in \S \ref{indDirtheor} induces a regular selfadjoint operator $\Dir_{\alpha}:=(\Dir_E)_{\alpha}$ on $H(\A) \ten E$ by Prop. \ref{isoop}. Since $a(\Dir_{\alpha}^2+1)^{-1} \in K(H(\A)\ten E)$ for any $a \in \A$ and $[\Dir_{\alpha},a] \in B(H(\A) \ten E)$ for any $a \in \Ai$, we obtain a class $[\Dir_{\alpha}] \in \KKi_{\dim G}(\A,\A^{\alpha})$. 

\begin{prop}
It holds that
$$\phi_*[\Dir_{\alpha}]=[\Dir_E]  \in \KKi_{\dim G}(\A,\A \times_{\alpha} G) \ .$$
\end{prop}

\begin{proof}
For $a \in \A_i^{\alpha}$ we have that $\phi(a)=\Theta_{a,e_{i+1}}$ 
under the identification $K(H(\A))\cong \A\times_{\alpha} G$. 
Recall that the $K(H(\A))$-valued scalar product on the Hilbert $K(H(\A))$-module $H(\A) \ten_{\phi} K(H(\A))$ is given by $\langle v \ten K_1,w \ten K_2\rangle_{\ten} = K_1^* \phi(\langle v,w \rangle)K_2$.
Furthermore if $v,w \in \A_i$, then $\phi(\langle v,w \rangle)=\Theta_{e_{i+1},v} \Theta_{w,e_{i+1}}$,
hence $$\langle v \ten K_1,w \ten K_2\rangle_{\ten}=K_1^*\Theta_{e_{i+1},v} \Theta_{w,e_{i+1}}K_2 \ .$$
This implies that the map $H(\A) \ten_{\phi} K(H(\A)) \to K(H(\A))$ that maps $v\ten K$ with $v \in \A_i$ to $\Theta_{v,e_{i+1}}K$ is an isometry. It is clear that the map is surjective. Hence it is an isomorphism.
The induced isomorphism $$(H(\A) \ten E)\ten_{\phi} K(H(\A)) \cong K(H(\A)) \ten E$$ interchanges $\Dir_{\alpha} \ten 1$ and $\Dir_E$ and is compatibel with the action of $\A$. This implies the assertion.
\end{proof}

For the class $[\Dir_{\alpha}] \in \KKi_{\dim G}(\A,\A^{\alpha})$ to be well-defined the condition on $\alpha$ to be saturated is not necessary: For example  it was indicated in \cite[\S 3]{ri} that the triple $(\A,\Dir_{\alpha},H(\A))$ is a spectral triple if the action is ergodic. However the relation to the construction in \S \ref{indDirtheor} is more involved in general.

Next we show that $\Tr_{\tau}$ from in \S \ref{spectrip} considered as a trace on $K(H(\A))$ behaves as one would expect from a trace on an algebra of compact operators. For notational simplicity we assume from now on that $E=\bbbc$.

\begin{prop} 
For $v,w \in \A_c \subset H(\A)$
$$\Tr_{\tau}(\Theta_{v,w})=\tau(\langle w,v\rangle_{H(\A)}) \ .$$
\end{prop} 

\begin{proof}
For $v,w \in \A_c$ there is $i \in \bbbn$ such that $e_ive_i=v,~e_iwe_i=w$. The constant functions $v,w \in C(G,\A)^{\st}\subset \A \times_{\alpha} G$ are in $l^2(\N)$, hence $\Theta_{v,e_i},\Theta_{w,e_i} \in l^2(\N)$. Thus $\Theta_{v,w}=\Theta_{v,e_i}\Theta_{e_i,w} \in l^1(\N)$
and $\Tr_{\tau}(\Theta_{v,w})=\Tr_{\tau}(\Theta_{v,e_i} \Theta_{e_i,w})$.

By eq. \ref{intop} the image of $\Theta_{v,e_i}$ under the map $K(H(\A)) \cong \A \times_{\alpha}G \to K(L^2(G,\A))$ is the integral operator with integral kernel $k_v(g,h):=\alpha_{g^{-1}}(v)$, and the image of $\Theta_{e_i,w}$ is $k_w(g,h)=\alpha_{g^{-1}}(\alpha_{gh^{-1}}(w^*))=\alpha_{h^{-1}}(w^*)$. This implies
\begin{align*}
\Tr_{\tau}(\Theta_{v,e_i} \Theta_{e_i,w}) &= \tau(\int_G\int_G k_v(g,h)k_w(h,g)~dhdg)\\
&=\tau(\int_G \int_G \alpha_{g^{-1}}(v)\alpha_{g^{-1}}(w^*)~dhdg)\\
&= \tau(\int_G \alpha_g(vw^*)~dg)\\
&=\tau(\langle w,v \rangle_{H(\A)}) \ .
\end{align*}
\end{proof}

Now we will establish the connection to the classical situation. Consider a compact principal bundle $G \to P \to B$ endowed with a vertical Riemannian metric coming from the Riemannian metric on $G$. The induced vertical $L^2$-scalar product is denoted by $\langle~,~\rangle^v$. For simplicity we assume that $P$ is a closed manifold. Let $\Delta_v$ be the vertical Laplace operator. For $s \in \bbbr$ we define the vertical Sobolev space $H_v^s(P)$ as the completion of $\C(P)$ with respect to the norm induced by the $C(B)$-valued scalar product 
$$\langle f_1,f_2\rangle_s^v(b):=\langle (\Delta_v+1)^sf_1, (\Delta_v+1)^sf_2 \rangle^v(b) \ .$$ Denote by $\Psi^s_v(P)$ the space of vertical classical pseudodifferential operators of order smaller or equal to $s$ with coefficients depending continuously on the parameter $b\in B$ and by $\ov{\Psi}^s_v(P)$ its completion with respect to the operator norm of bounded operators from $H_v^s(P)$ to $H_v^0(P)$. Furthermore we define $$\Psi_{\alpha}^s(G,C(P)):=\U(\alpha)\Psi^s(G,C(P))^{\Alpha}\U(\alpha)^{-1} \  .$$
By Prop. \ref{isoop} the elements of $\Psi_{\alpha}^s(G,C(P))$ are continuous operators form $H_{\alpha}^s(C(P))$ to $H_{\alpha}^0(C(P))$. Let $\ov{\Psi}_{\alpha}^s(G,C(P))$ be the completion with respect to the operator norm.

\begin{prop}
For all $s \in \bbbr$
$$H^s(C(P))=H_v^s(P)$$  and 
$$\ov{\Psi}^s_v(P) =\ov{\Psi}_{\alpha}^s(G,C(P)) \ .$$
\end{prop} 

\begin{proof}
Let $X \in \gl$ and let $X_v$ be the induced invariant vertical vector field on $P$.

From $X_{\alpha}f=X_vf$ for $f \in \C(P)$ it follows that $$(\Delta_v+1)f=(\Delta_{\alpha}+1)f \ .$$
Now the first assertion follows from the definition. 

It is enough to show the second assertion for $s=0$ since $\ov{\Psi}^s_v(P)=\ov{\Psi}^0_v(P)(\Delta_v+1)^{s}$, and analogously for $\ov{\Psi}^s_{\alpha}(G,C(P))$. We denote by $K(L^2_v(P))$ the space of compact vertical operators depending continuously on the base point. It is not difficult to check that it is isomorphic to $K(H(C(P))$.

We write $S_vP\subset T_vP$ for the vertical sphere bundle. An orthonormal basis $(X_1, \dots, X_{\dim G})$ of $\gl$ defines a trivialisation of $TG$, and also of $T_vP$. Thus $S_vP\cong S^{\dim G-1} \times P$. We write $\sigma_v:\ov{\Psi}^0_v(P) \to C(S_vP)$ for the vertical symbol map.
 
There is the following commutative diagram:
$$\xymatrix{
0 \ar[r]& K(L^2_v(P)) \ar[r]\ar[d]^=& \ov{\Psi}^0_v(P) \ar[r]^{\sigma_v}\ar[d]& C(S^{\dim G-1} \times P) \ar[r]\ar[d]& 0  \\
0 \ar[r]& K(H(C(P))) \ar[r]& B(H(C(P))) \ar[r]& B(H(C(P)))/K(H(C(P))) \ar[r]& 0 \\
0 \ar[r]& K(H(C(P))) \ar[r]\ar[u]^=& \ov{\Psi}^0_{\alpha}(G,C(P)) \ar[r]^{\sigma}\ar[u]& C(S^{\dim G-1} \times P) \ar[r]\ar[u] & 0 }$$
The rows are exact. All vertical maps are injective. 

In the following we show that the images in $B(H(C(P)))/K(H(C(P)))$ of the two vertical maps in the last column agree. Then the images in $B(H(C(P)))$ of the vertical maps in the second row also agree. This implies the assertion.

Let $(X^i)_{i=1, \dots, \dim G}$ be an orthonormal basis of $\gl$ and let $x_i:\gl \cong \bbbr^{\dim G} \to \bbbr$ be the induced coordinate functions. We use the same notation for the restriction of $x_i$ to  $S^{\dim G-1}$. For $f \in C(P)$ the operator $fX^i_v (1+\Delta_v)^{-1/2}$ is in $\ov{\Psi}^0_v(P)$ and the operator $fX^i_{\alpha}(1+\Delta_{\alpha})^{-1/2}$ in $\ov{\Psi}^0_{\alpha}(G,C(P))$. Furthermore $$\sigma_v(fX^i_v (1+\Delta_v)^{-1/2})= f x_i = \sigma(fX^i_{\alpha}(1+\Delta_{\alpha})^{-1/2}) \ .$$ Since the algebra generated by the functions $fx_i$with $f \in C(P)$ and $i=1 \dots \dim G$ is dense in $C(S^{\dim G-1} \times P)$ and since the operators $fX^i_{\alpha}(1+\Delta_{\alpha})^{-1/2}$ and $fX^i_v (1+\Delta_v)^{-1/2}$ agree in $B(H(C(P))$, the assertion follows.
\end{proof}

If $P$ is endowed with a Riemannian metric and has unit volume, then an invariant trace $\tau$ on $\A=C(P)$ is given by $\tau(f)=\int_P f$. In this situation Prop. \ref{indtheo} corresponds to the zero degree part of the Atiyah-Singer index formula for families. 

In order to complete the picture, we establish a relation between principal and saturated action. This is not needed in the remainder of the paper.
The notion of a principal action was introduced in \cite{e} in the more general framework of Hopf $C^*$-algebras. Specialized to the present context its definition is:

\begin{ddd} Let $\alpha$ be an action of a compact group $G$ on a $C^*$-algebra $\A$. Then $\alpha$ is called principal if the span $\Fi$ of the functions $a\alpha(b^*),~a,b \in \A$ is dense in $C(G,\A)$. 
\end{ddd}

If $\alpha$ is principal, then $\alpha$ is saturated. The converse holds at least for compact Lie groups:

\begin{prop}
\label{satprinc}
Let $G$ is a compact Lie group and let $\alpha$ be an action of $G$ on a $C^*$-algebra $\A$.
If $\alpha$ is saturated, then $\alpha$ is principal.
\end{prop}

\begin{proof}
We denote by $\tilde{\A}$ the unitalization of $\A$ if $\A$ is nonunital, and we set $\tilde \A=\A$ for unital $\A$ in the following. The action $\alpha$ extends to $\tilde{\A}$.

We need some technical preliminaries:
 
We write $C_{\alpha}(G,\tilde \A)$ for the space $C(G,\tilde \A)$ endowed with the right $\A$-module structure 
$$C_{\alpha}(G, \tilde \A) \times \A \to C_{\alpha}(G,\tilde \A),~ (f,a) \mapsto f\alpha(a) \ .$$

The Hilbert $\tilde \A$-module $L^2_{\alpha}(G,\tilde \A)$ is defined as the completion of $C_{\alpha}(G,\tilde \A)$ with respect to the norm induced by the $\tilde \A$-valued scalar product 
$$\langle f_1,f_1 \rangle_{\alpha}=\int_G \alpha_g^{-1}(f_1(g)^*f_2(g))~dg \ .$$
There is an isometric isomorphism of Hilbert $\tilde \A$-modules $U:L^2(G,\tilde \A) \to L^2_{\alpha}(G,\tilde \A),~f\mapsto (g \mapsto \alpha_gf(g))$. 

Since for $f_1,f_2 \in C(G,\A)$ 
$$UC(f_1)U^{-1}f_2= f_1 \st f_2 \ ,$$
the map $f_1 \mapsto (f_2 \mapsto f_1 \st f_2))$ extends to a $C^*$-homomorphism $\A \times_{\alpha} G \to B(L^2_{\alpha}(G,\tilde \A))$.  

Let $\Delta$ be the Laplace operator on $G$. For $t>0$ 
$$Ue^{-t\Delta}U^{-1}:L^2_{\alpha}(G,\tilde \A) \to C(G,\tilde \A)$$ 
is continuous.
We denote by $k_t$ the integral kernel of $e^{-t\Delta}$.  
Recall that $e^{-t\Delta}$ converges strongly to the identity on $C(G,\A)$ for $t  \to 0$.

For $f \in C(G,\A)$
\begin{align*}
f \st k_t(\cdot,e)&=\int_G k_t(\cdot,h)f(h)~dh \\
&= e^{-t\Delta}f \ .
\end{align*}
We see that $f \st k_t(\cdot,e)$ converges to $f$ in $C(G,\A)$ for $t \to 0$.

If $f=a\alpha(b^*)$ for $a,b\in \A$, then $$a\int_G k_t(g,h)\alpha_h(b^*)~dh=a\alpha_g\int_G k_t(e,h) \alpha_h(b^*)~dh \ ,$$
thus $f \st k_t(\cdot,e) \in \Fi$. 

Furthermore for $f \in C(G,\A)$ 
\begin{align*}
(Ue^{-t\Delta}U^{-1}f)(g)&= U\int_G k_t(g,h) \alpha_{h^{-1}}f(h)~dh\\
&=\int_G k_t(gh^{-1},e)\alpha_{gh^{-1}}f(h)~dh\\
&= \int_G k_t(h,e) \alpha_h f(h^{-1}g)~dh\\
&= (k_t(\cdot,e) \st f)(g) \ .
\end{align*}

Since for $f=a\alpha(b^*)$ with $a,b \in \A$
$$(k_t(\cdot,e) \st f)(g)=\int_G k_t(h,e) \alpha_h(a) ~dh ~\alpha_g(b^*)  \ ,$$
it follows that $Ue^{-t\Delta}U^{-1}(\Fi) \subset \Fi$. 

It is clear that $Ue^{-t\Delta}U^{-1}f$ converges to $f$ in $C(G,\A)$ for $t \to 0$ for general $f \in C(G,\A)$.

We are ready for the main argument:

Let $f \in C(G,\A)$ and let $(f_n)_{n \in \bbbn}$ be a sequence in $\Fi$ converging in $\A \times_{\alpha} G$ to $f$. 

Let $\ve >0$. 

There is $t>0$ such that 
$$\|Ue^{-t\Delta}U^{-1}(f \st k_t(\cdot,e))- f\|_{\infty} < \ve/2\ ,$$ where $\|~\cdot ~\|_{\infty}$ denotes the supremum norm.

Since $k_t(\cdot,e) \in L^2_{\alpha}(G,\tilde \A)$, the functions $f_n \st k_t(\cdot,e)$ converge to $f \st k_t(\cdot,e)$ in $L^2_{\alpha}(G,\tilde \A)$ for $n \to \infty$. Hence there is $n \in \bbbn$ such that $$\|Ue^{-t\Delta}U^{-1}(f_n \st k_{t}(\cdot,e)) - Ue^{-t\Delta}U^{-1}(f \st k_{t}(\cdot,e))\|_{\infty} < \ve/2 \ ,$$ 
hence $$\|f-Ue^{-t\Delta}U^{-1}(f_n \st k_{t}(\cdot,e))\|_{\infty} < \ve \ .$$
From the first part of the proof it follows that $Ue^{-t\Delta}U^{-1}(f_n*k_{t}(\cdot,e)) \in \Fi$. 
\end{proof}

In order to determine whether a torus action is saturated or not the following criterium due to Rieffel \cite[Theorem 7.1.15.]{ph} is useful:

\begin{prop}
\label{largsat}
Let $G$ be a compact abelian group and let $\alpha$ be a $G$-action on $\A$. Let $\hat G$ be the dual group of $G$. For $\chi \in \hat G$ let $$\A(\chi)=\{ a \in \A~|~ \alpha_g(a)=\chi(g)a \mbox{ for all } g\in G\} \ .$$
Then $\alpha$ is saturated if and only if the linear span $\A(\chi)^*\A(\chi)$ of elements of the form $a^*b$ with $a,b \in \A(\chi)$ is dense in $\A^{\alpha}$ for any $\chi \in \hat G$. 
\end{prop}

This implies that $\A^{\alpha}$ is saturated if $\ov{\A(\chi)^*\A(\chi)}$ generates $\A^{\alpha}$ as a $\A^{\alpha}$-bimodule for any $\chi \in \hat G$.

\section{Examples of saturated actions by compact Lie groups}
\label{examples}

\subsection{Principal bundles}

Let $\A=C_0(X)$ where $X$ is a locally compact space. By \cite[Prop. 7.1.12]{ph} if the action $\alpha$ is saturated, then the action of $G$ on $X$ is free. Since $G$ is compact, the action is proper. Hence the fibration $X \mapsto X/G$ is a principal bundle. We refer to \cite{bhms} for a detailed survey on the relation between noncommutative notions and their commutative counterparts in this context.

\subsection{Crossed products by $\bbbz^n$-actions}

If $\beta$ is an action of $\bbbz^n$ on a $C^*$-algebra $\B$, then there is an induced action $\alpha$ of the dual group $T^n$ on the crossed product $\A=\B \times_{\beta}\bbbz^n$. It is well-known that this action is saturated.

\subsection{Noncommutative tori}

For notational simplicity we restrict to noncommutative two-tori in the following. 

The noncommutative two-torus $\A_{\theta}$ with angle $2 \pi \theta,~ \theta \in [0,1)$ is defined as the universal algebra generated by two unitaries $U_1,U_2$ subdued to the relation $U_1U_2= e^{2\pi i\theta}U_2U_1$. If $\theta \notin \bbbq$, then $\A_{\theta}$ is also called an irrational rotation algebra. 

There is an action of $T^2$ on $\A_{\theta}$ determined by $(z_1,z_2,U_i) \mapsto z_iU_i$. In \cite{e} it was verified that the action is principal. As an illustration we give a different argument based on Prop. \ref{largsat}, which works in similar form for other universal algebras as well, for example for graph $C^*$-algebras (see the references below).  

\begin{prop}
The action of $T^2$ on $\A_{\theta}$ is saturated.
\end{prop}

\begin{proof}
For $(k,l) \in \bbbz^2$ we have that $U_1^kU_2^l \in \A_{\theta}(k,l)$. Since $(U_1^kU_2^l)^*(U_1^kU_2^l)=1$, the assertion follows from Prop. \ref{largsat}. 
\end{proof}

\subsection{Higher rank graph $C^*$-algebras}

Higher rank graph $C^*$-algebras were introduced in \cite{kp} as a generalization of graph $C^*$-algebras. The notions used here are as in \cite{prs}. See \cite{r} for a monograph on graph $C^*$-algebras, which also includes a discussion on the higher rank case. A $k$-graph $C^*$-algebra comes equipped with an $T^k$-action, called gauge action.

Higher rank graph $C^*$-algebras behave similar to Cuntz-Krieger algebras. In \cite[Lemma 4.1.1]{pr} it was shown that the action on row-finite Cuntz-Krieger algebras is saturated under certain conditions. The proof is based on Prop. \ref{largsat} and uses a particularly simple set of generators of the fixed point algebra. It carries over to row-finite graph $C^*$-algebras, where a similar description of the fixed point algebra exists, see \cite[Cor. 3.3]{r}. The proof can also be generalized to the higher rank case by using the description of the fixed point algebra from \cite[Lemma 3.3 (ii)]{kp}.
It turns out that the right condition in the case of $k$-graph $C^*$-algebras is that the $k$-graph have no sinks nor sources. 

Since we did not find the statement in the literature, we formulate it as a proposition here:

\begin{prop}
Let $E$ be a row-finite $k$-graph with no sources and no sinks and let $C^*(E)$ be its $k$-graph $C^*$-algebra. Then the gauge $T^k$-action on $C^*(E)$ is saturated. 
\end{prop}

For locally finite graphs the previous proposition was proven in \cite{sz} by verifying the definition for principal actions.  

One may ask which geometric properties of the graph are reflected by $C^*$-algebraic properties of gauge action on $C^*(E)$. Using the dual Pimsner-Voiculescu sequence one can, under certain conditions, get back the length of loops in the graph. See \cite{pre}\cite{clm} for more detailed discussions of some aspects of this question. The action (up to equivariant isomorphism) does not determine the graph up to isomorphism: By the proof of \cite[Cor. 2.6]{r} there is an isomorphism between the $C^*$-algebra of a graph and the $C^*$-algebra of the dual graph that preserves the gauge action. However a graph need not be isomorphic to its dual graph.

\subsection{Compact quantum groups and quantum principal bundles}
We refer to \cite{md} for a survey on compact quantum groups, which were introduced by Woronowicz. We give a short introduction with the aim to show that the action of a subgroup on a compact quantum group is saturated. 

A comultiplication on a unital $C^*$-algebra $\A$ is a
unital *-homomorphism $\Phi:\A \to \A \ten \A$ satisfying coassociativity $(\Phi \ten 1)\Phi = (1 \ten \Phi)\Phi$. A compact quantum group is a pair $(\A,\Phi)$, where $\A$ is a unital $C^*$-algebra and $\Phi$ is a comultiplication on $\A$ such that $\Phi(\A)(\A \ten 1)$ and $(\A \ten 1)\Phi(\A)$ are dense in $\A \ten \A$. 
A compact quantum subgroup $(\B,\Phi_2)$ of a compact quantum group $(\A,\Phi)$ is a surjective $C^*$-homomorphism $\pi:\A \to \B$ compatible with the coproducts. 

A compact Lie group defines a compact quantum group $(C(G),\Phi_G)$ with $\Phi_G:C(G) \to C(G\times G),~(\Phi_Gf)(gh)=f(gh)$.

Let $G$ be a compact Lie group and a subgroup of a compact quantum group $(\A,\Phi)$. The surjective homomorphism $\pi:\A \to C(G)$ induces a homomorphism $\alpha=(\pi \ten 1)\circ\Phi:\A \to C(G,\A)$, hence an action of $G$ on $\A$.
The span of $(1\ten b)\Phi(a),~a,b \in \A$ is dense in $\A \ten \A$ and the range of $(\pi \ten 1):\A \ten \A \to C(G,\A)$ is dense. Thus the equality $(\pi \ten 1)(((1\ten b)\Phi(a))=b \alpha(a)$ implies that the span of $b\alpha(a),~a,b \in \A$ is dense in $C(G,\A)$, hence also in $\A \times_{\alpha} G$. This shows that $\alpha$ is saturated. 

Quantum principal bundles have been extensively studied in the literature in various settings. Our results apply to quantum principal bundles in the sense of \cite{e} whose structure quantum groups are compact Lie groups. Examples can be found in \cite{hms1}\cite{hms2} \cite{hs1} \cite{hs2} \cite{sz}, among others.

\textsc{Leibniz-Archiv\\
Waterloostr. 8\\
30169 Hannover\\
Germany} 

\textsc{Email: ac.wahl@web.de}

\end{document}